\documentclass[12pt,onecolumn]{article}

\usepackage{tikz}

\usepackage{amsmath,amssymb,amsfonts,theorem}
\usepackage{graphics,color}
\usepackage{color}
\usepackage{subfigure}
\graphicspath{{./Figures/}}
\usepackage{graphicx}
\usepackage{fullpage}

{ \theorembodyfont{\normalfont} 
\newtheorem{example}{Example}
\newtheorem{remark}{Remark}
}
\newtheorem{assumption}{Assumption}

\newtheorem{theorem}{Theorem}
\newtheorem{lemma}{Lemma}
\newtheorem{corollary}{Corollary}

\newcommand{\sat}{\textrm{sat}}
\newcommand{\sign}{\textrm{sign}}

\newcommand{\R}{\mathbb{R}}

\newcommand{\dst}{\displaystyle}

\def\be{\begin{equation}}
\def\ee{\end{equation}}
\def\ba{\begin{array}}
\def\ea{\end{array}}
\def\eqa{\begin{eqnarray}}
\def\eqe{\end{eqnarray}}

\begin{document}


\title{Balancing time-varying demand-supply in distribution networks: an internal model approach}
\author{Claudio De Persis\thanks{C.~De Persis is with ITM, Faculty of Mathematics and
Natural Sciences (FWN), University of Groningen, 9747 AG Groningen, The
Netherlands, {\tt\small c.de.persis@rug.nl} and Department of Computer, Control and Management Engineering, Sapienza Universit\`a di Roma, Italy. This research is partially supported by an FWN starting grant. 
}%
}

\date{}
\maketitle

\begin{abstract}
The problem of load balancing in a distribution network under unknown time-varying demand and supply is studied.  A set of distributed controllers  which regulate the amount of flow through the edges is designed to guarantee convergence of the solution to the steady state solution. The results are then extended to a class of nonlinear systems and compared with existing results. Incremental passivity and internal model are the main analytical tools. 
\end{abstract}

\section{Introduction}

Cooperative control systems have been widely investigated in a variety of different contexts \cite{JNT:DB:MA:TAC86, LM:TAC04, stan.sepulchre.tac07, bai.et.al.book}. Less attention has been devoted to cooperative control in the framework of dynamical flow networks, with some  interesting exceptions \cite{MB:DZ:FA:CDC11, BETAL:CDC10, DB:FB:RP:TAC10, DB:FB:RP:JOTA09, AJVDS:JW:LHMNLC12, MB:FB:ACC11}. The aim of this paper is to study a class of cooperative control algorithms in the context of distribution networks under exogenous inputs. \\
{\it Main contribution.}  We analyze and design distributed controllers at the edge which achieve load balancing in the presence of time-varying demand and supply (exogenous signals). The role of internal model and incremental passivity is investigated for the problem at hand. Similar tools have been used for controlled synchronization and leader-follower formation control in e.g.~\cite{wieland.et.al.aut11, bai.et.al.book, stan.sepulchre.tac07, CDP:BJ:CDC12} and references therein. We address a different problem and we tackle it in a novel way. The load distribution problem is then considered for a more general class of systems and this allows us to make a comparison with the results of \cite{arcak.tac07} and \cite{MB:DZ:FA:CDC11}. 

{\it Literature review.} The literature on the control of flow or distribution networks is  wide and multi-disciplinary. Here we restrict ourselves to a very small portion of it, focusing on a model which takes into account the amount of stored material at the nodes and mass balance. This class of systems has been used to model data networks \cite{MS:tac82} and  supply chains \cite{AGT:TCST11} for instance. Our paper focuses on the problem of stabilizing the flow network to a steady state solution in the presence of exogenous time-varying demand and supply under the scenario in which the controllers aim at a uniform distribution of the material among the nodes. This is a problem which has attracted considerable attention in the area of  parallel and distributed computation \cite{JNT:DB:MA:TAC86} and has been recently reconsidered for instance in \cite{MB:FB:ACC11} where input and state constraints have been taken into account and a connection with \cite{LM:TAC04} has been established. The work \cite{MB:FB:ACC11} did not consider the presence of external inputs. A large amount of work on the topic of flow control in the presence of disturbances has been carried out in works such as \cite{BETAL:CDC10, DB:FB:RP:TAC10, DB:FB:RP:JOTA09} where the problem is cast in the robust control framework. The approach in our paper is based on the theory of output regulation and to the best of our knowledge this has not been considered before. A similar problem has been tackled in \cite{AJVDS:JW:LHMNLC12} but the authors restrict themselves to the class of constant disturbances.

The organization of the paper is as follows. The class of systems under study is introduced in Section \ref{sec1}, the design of the edge regulators is carried out in Section \ref{sec2} and the extension to a class of nonlinear system in Section \ref{sec3}. The conclusions are discussed in the last section.  

\section{Distribution networks and demand supply balancing}\label{sec1}

Consider the system
\be\label{flow.net}
\ba{rcl}
\dot x &=& B\lambda + P d
\ea\ee
with $x\in \R^n$ the state, $\lambda\in \R^m$ the control vector and $d\in \R^q$, $q\le n$,  a disturbance vector. The $(n\times m)$ matrix $B$ is the incidence matrix of an {\em undirected} graph $G=(V,E)$ where $|V|=n$, $|E|=m$. The ends of the edges of $G$ are labeled with a `+' and a `-'. Then 
\[
b_{ik}=\left\{\ba{ll}
+1 & \textrm{$i$ is the positive end of $k$}\\
-1 & \textrm{$i$ is the negative end of $k$}\\
0 & \textrm{otherwise}
\ea\right.
\]

The system above is a simple model of a  flow network \cite{BETAL:CDC10} and it has been used also to model data networks \cite{MS:tac82} and  supply chains \cite{AGT:TCST11}.
The state $x_i\in \R$, $i\in {\cal I}:=1,2,\ldots,n$ represents the quantity of material stored at the node $i$, $\lambda_k\in \R$, $k=1,2,\ldots,m$ the flow through  the edge $k$. The disturbance $d_j\in \R$ represents the inflow or the outflow at some node. \\
The available measurements are the differences among the quantities stored  at the nodes namely,
 $z=B^T x$.\\
We assume that each disturbance $d_j$ is supposed to be generated by the exosystem
\[\ba{rcll}
\dot w_{j} &=&  S_j^d w_{j}& \\
d_j &=&  \Gamma^d_j w_{j},& j=1,\ldots, q,
\ea\]
where $w_{j}\in \R^{p_j}$ is the state of the exosystem which describes the evolution of the inflow/outflow $j$ and $\Gamma^d_j, S_j^d$ are suitable matrices. Considering more general classes of exosystems is left for future research. 
We give the system above the compact form
\be\label{exos}\ba{rcll}
\dot w &=&  S^d w & \\
d &=&  \Gamma^d w, &
\ea\ee
where
$w=(w_{1}^T\ldots w_{q}^T)^T$, $d=(d_{1}^T\ldots d_{q}^T)^T$, $S^d={\rm block.diag}(S_1^d,\ldots,  S_q^d)$,
$\Gamma^d={\rm block.diag}(\Gamma_1^d,\ldots, \Gamma_q^d)$.
The model (\ref{flow.net}) and the  overall exosystem (\ref{exos}) return  the closed-loop system
%
\be\label{closed.loop.system}
\ba{rcl}
\dot w &=& S^d w\\
\dot x &=& B\lambda + Pw\\
z &=& B^T x
\ea
\ee
where by a slight abuse of notation we renamed $P\Gamma^d$ simply as $P$.


We are interested in the problem of distributing the cumulative imbalance of the network due to  the in- and out-flow among the nodes. More formally the problem at hand is as follows:

\noindent {\bf Load balancing at the nodes}
Find  distributed
dynamic feedback control laws
\be\label{dfc}\ba{rcl}
\dot \eta_k &=& \Phi_k \eta_k +\Lambda_k z_k\\
\lambda_k &=& \Psi_k \eta_k+\Gamma_k z_k,\;\; k=1,\ldots, m
\ea
\ee
such that, for each initial condition $(w_0, x_0,\eta_0)$,  the solution of the closed-loop system (\ref{closed.loop.system}), (\ref{dfc}) satisfies $\lim_{t\to +\infty} z(t)=0$.

In what follows we propose a solution to the problem.

\section{Design of regulators at the edges}\label{sec2}

We focus on flow networks whose underlying graph satisfies the following standing assumption:
\begin{assumption}\label{connectivity}
The graph $G$ is connected.
\end{assumption}


The first  result concerns the characterization of a ``steady state" solution to the problem:
\begin{lemma}\label{l1}
Let Assumption \ref{connectivity} hold. For each $w$ solution to $\dot w = S^d w$, if there exist  a function $\lambda_w:\R_+\to \R^m$  and a  continuously differentiable function $x^w:\R_+\to \R^n$ solution to
\be\label{re.eq1}\ba{rcl}
\dot x^w &=& B\lambda_w + Pw\\
\ea\ee
and
\be\label{re.eq2}\ba{rcl}
0 &=& B^T x^w
\ea\ee
then 
\be\label{cum.imb}
x^w=\mathbf{1}_n x^w_\ast,\quad
\dot x^w_\ast =\frac{\mathbf{1}_n^T P w}{n}
\ee
and $\lambda_w = M w$,
for some matrix M. If the graph is a tree, then the matrix $M$ is unique.
\end{lemma}


\begin{proof}
From Assumption \ref{connectivity} and (\ref{re.eq2}),  one obtains that $x^w=\mathbf{1}_n x^{w}_\ast$, for some function $x^w_\ast:\R_+\to \R$. Replacing the expression  of $x^w$ in 
(\ref{re.eq1}) one has
\be\label{sle0}
\mathbf{1}_n \dot x^{w}_\ast = B\lambda_w + Pw,
\ee
As ${\cal N}(B^T)={\cal R}(\mathbf{1}_n)$, with ${\cal N}, {\cal R}$ the null space and the range of a matrix, multiplying on the left both sides of (\ref{sle0}) by $\mathbf{1}_n^T$ yields
$\dot x^{w}_\ast =  \frac{\mathbf{1}_n^T Pw}{n}$
as claimed. Replace the latter in (\ref{sle0}) to obtain
$Y P w =B\lambda^w$, $Y=\mathbf{1}_n\frac{\mathbf{1}_n^T}{n}-I_n$.
%
By Assumption \ref{connectivity}
and without loss of generality (up to a relabeling of the edges of the graph), $B\lambda^w= B_a \lambda^w_a+ B_b \lambda^w_b$ with $B_a$ full-column rank and $\lambda^w_a\in \R^{n-1}$. If a solution $\lambda^w_a$ to $Y P w =B\lambda^w$ exists, then 
$\lambda^w_a= (B_a^TB_a)^{-1}B_a^T(Y P w-B_b \lambda^w_b)$.
Letting $\lambda^w_b=0$ one obtains $\lambda^w_a= M_a w$, with $M_a=(B_a^TB_a)^{-1}B_a^T Y P$. \\
If $G$ is  a tree, then $B$ is full-column rank and  $\lambda^w= (B^T B)^{-1}B^TY P w$.
%
\end{proof}

In what follows, we assume that a solution to (\ref{re.eq1}), (\ref{re.eq2}) exists. Moreover, if $m>n-1$, then without loss of generality we assume that the first $n-1$ columns of $B$ are linearly independent and we let the last $m-n+1$ components of $\lambda^w$ be identically zero. 


\vspace{-0.25cm}

\begin{remark}  From (\ref{cum.imb}), by integration, one has
\[
x^w(t)= \mathbf{1}_n\left(x^w_\ast(0)+ \dst\int_0^t \frac{\mathbf{1}_n^T P w(s)}{n} ds\right).
\]
Observe that $x^w$ depends on the initial condition and strictly speaking cannot be referred to as a steady state solution.   
Bearing in mind the interpretation of (\ref{flow.net}) as a flow network and of $Pw$ the vector of the inflows and outflows of the network,  the integral $\int_0^t \frac{\mathbf{1}_n^T P w(s)}{n} ds$ can be seen as the {\em cumulative} imbalance of the network. In other words, if for any given $w$ a solution to the load balancing problem exists, then the state at each node equals -- up to a constant -- the cumulative imbalance of the network. \\
In the case of a network with no imbalance, i.e.~$\mathbf{1}_n^T P w(t)=0$ for all $t\ge 0$,  $x^w$ is a {\em constant} vector.
\end{remark}

\vspace{-0.25cm}

\begin{example}\label{ex1} Consider the graph depicted in Fig.~\ref{graph.example}. The graph corresponds to system (\ref{flow.net}) with
\[
B=
\left(\ba{rrr}
-1 & 0 & 1\\
1 & -1 & 0\\
0 & 1 & -1
\ea
\right),
\quad
P=
\left(\ba{rrr}
1 & 0 \\
0 & -1\\
0 & 0
\ea
\right)
\]
The solutions of (\ref{re.eq1})-(\ref{re.eq2}) (with $w=d$) are as follows
\[\ba{rcl}
\dot x^w_\ast &=& \frac{d_1-d_2}{3}\\
\lambda^w_1 &=& \lambda^w_3+\frac{2d_1+d_2}{3}\\
\lambda^w_2 &=& \lambda^w_3+\frac{d_1-d_2}{3}.
\ea\]
A solution is obtained letting  $\lambda^w_3=0$.
%
%
\begin{figure}
\begin{center}
\begin{tikzpicture}
  [scale=.8,auto=left,every node/.style={circle,fill=blue!20}]
  \node (n1) at (1,4) {1} 
  node[draw=none,fill=none] at (1.7,4.3) {$-$}
  node[draw=none,fill=none] at (3,4.3) {$1$} 
  node[draw=none,fill=none] at (4.3,4.3) {$+$};
  \node (n2) at (5,4)  {2}
node[draw=none,fill=none] at (5,3.4) {$-$}
node[draw=none,fill=none] at (4.3,2.5) {$2$}
 node[draw=none,fill=none] at (3.7,1.4) {$+$};
   \node (n3) at (3,1)  {3}
node[draw=none,fill=none] at (2.4,1.4) {$-$}
node[draw=none,fill=none] at (1.6,2.5) {$3$}
node[draw=none,fill=none] at (1,3.4) {$+$};

  \foreach \from/\to in {n1/n2,n2/n3,n1/n3}
    \draw (\from) -- (\to);
    \draw [->] (-1,4) -- (0.5,4) node[draw=none,fill=none] at (-1.5,4) {$d_1$};
    \draw [->] (5.5,4) -- (7,4) node[draw=none,fill=none] at (7.5,4) {$d_2$};

\end{tikzpicture}
\end{center}
\caption{The distribution network considered in the Example \ref{ex1}. \label{graph.example}}
\end{figure}
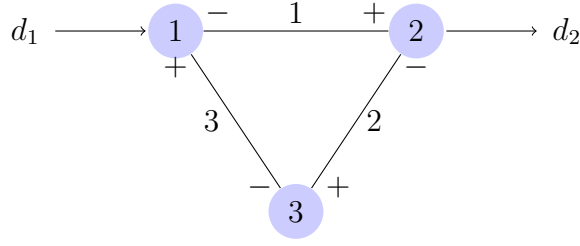
\end{example}

We introduce now a system which generates the control signal $\lambda^w$ in Lemma \ref{l1}.
Consider the input $\lambda^w_{k}$ associated with the edge $k$, with $k=1,2,\ldots, m$. In general, such input may depend on all the components of the disturbance vector $w$. Hence, to generate $\lambda^w_{k}$, the following system is proposed:
\be\label{im}
\ba{rcl}
\dot \eta_k &=& S^d \eta_k\\
 u_{k} &=& H_k \eta_k
\ea\ee
The statement below is immediate.

\vspace{-0.25cm}

\begin{lemma}\label{l3}
For any $w$ solution to $\dot w = S^d w$, there exists a solution $\eta^w_k$ to  (\ref{im}) such that $H_k \eta^w_k(t)=   \lambda_{k}^w(t)$  for all $t\ge 0$, where $\lambda^w_{k}$ is the $k$th entry of $\lambda^w$ in Lemma \ref{l1}.
\end{lemma}

\vspace{-0.25cm}

\begin{proof}
Choose $w(0)$ as the initial condition of (\ref{im}), then $\eta_k(t)=w(t)$ for all $t\ge 0$. As $\lambda^w = M w$, then it suffice to choose $H_k$ as the $k$th row of $M$ to have $H_k \eta_k^w(t)=H_k w(t)=   \lambda^{w}_{k}(t)$  for all $t\ge 0$.
\end{proof}

\vspace{-0.25cm}

\begin{remark}\label{rem.H0}
From the proof of Lemma \ref{l1} it turns out that $m-n+1$ components of $\lambda^w$ can be chosen identically zero. The matrices $H_k$ corresponding to these components are then identically zero as well. Hence, for $k=n, n+1, \ldots, m$, the system (\ref{im})  reduces trivially to $u_k=0$.
\end{remark}

The system (\ref{im}) is completed by adding control inputs $v_{k1}, v_{k2}$ to be designed for guaranteeing that the response of the closed-loop system converges to the desired response for $x$. Hence, we set
\be\label{im.complete}
\ba{rcl}
\dot \eta_k &=& S^d \eta_k + v_{k1}\\
 u_{k} &=& H_k \eta_k+ v_{k2},\quad k=1,2,\ldots, n-1
\ea\ee
with $\eta_k, v_{k1}\in \R^q$, $v_{k2} \in \R$, and $u_k=v_{k2}$ for $k=n, n+1, \ldots, m$.
\\
We write (\ref{im.complete}) in the form
\be\label{im.compact}
\ba{rcl}
\dot \eta  &=& \overline S \eta + v_{1}\\
\lambda &=& \overline H \eta+ v_{2}
\ea\ee
where $\eta=(\eta_1^T\;\eta_2^T\ldots\eta_{n-1}^T)^T$, $\overline S = I_{n-1} \otimes S^d$, where $\otimes$ denotes the Kronecker product,  and
\[
\overline H=\left(\ba{c}
\overline H_1\\ \mathbf{0}
\ea\right),\quad 
\overline H_1={\rm block}.{\rm diag}(H_1,\ldots,H_{n-1}).
\]
Observe that by Lemma \ref{l3}, for any $w$ and provided that $v_1=\mathbf{0}$, $v_2=\mathbf{0}$,  there exists a solution $\eta^w$ to  (\ref{im.compact}) which satisfies
\be\label{im.compact.w}
\ba{rcl}
\dot \eta^w  &=& \overline S \eta^w \\
\lambda^w &=& \overline H \eta^w
\ea\ee

\begin{theorem}\label{th1}
Consider the system (\ref{flow.net}), where $B$ is the incidence matrix of a graph $G$ and $d$ is a disturbance generated by the system (\ref{exos}). 

Under Assumption \ref{connectivity}, provided that $S^d_j$ is skew symmetric for each $j=1,2,\ldots, q$, the dynamic feedback controller (\ref{im.compact}) with $v_1=- \overline{H}^T B^T x$ and $v_2=-B^T x$, namely
\be\label{ix}
\ba{rcl}
\dot \eta  &=& \overline S \eta - \overline{H}^T B^T x\\
\lambda &=& \overline H \eta-B^T x
\ea\ee
guarantees boundedness of the state of the closed-loop system and asymptotic convergence of $x(t)$ to $\mathbf{1}_n (c' +\int_0^t \frac{\mathbf{1}_n^T Pw(s)}{n} ds)$ for some constant $c'$.
\end{theorem}

\begin{proof}
Consider the overall closed-loop system
\[
\ba{rcl}
\dot w &=& S^d w\\
\dot x &=& B(\overline H \eta+ v_{2}) + Pw\\
\dot \eta  &=& \overline S \eta + v_{1}\\
z &=& B^T x
\ea
\]
Introduce the new variables $\tilde x= x-x^w$, $\tilde \eta =\eta-
\eta^w$.
These satisfy
\[\ba{rcl}
\dot{\tilde x} &=& B(\overline H \eta+ v_{2}) + Pw-B\lambda^w-Pw\\
&=& B\overline H \tilde \eta+B(\overline H \eta^w-\lambda^w)+ B v_{2} \\
&=& B\overline H \tilde \eta+ B v_{2} \\
\ea\]
and
\[\ba{rcl}
\dot{\tilde \eta} &=& \overline S \eta + v_{1}-\overline S \eta^w \\
&=& \overline S \tilde \eta + v_{1}.
\ea\]
Introduce the Lyapunov function
$V(\tilde x, \tilde \eta) =\frac{1}{2}\left(\tilde x^T \tilde x+ \tilde \eta^T \tilde
\eta\right)$.
The function $V$ computed along the solutions of system
\be\label{cls.new.coord}\ba{rcl}
\dot{\tilde x} &=&  B\overline H \tilde \eta+ B v_{2} \\
\dot{\tilde \eta} &=& \overline S \tilde \eta + v_{1}
\ea\ee
satisfies
$\dot{V}(\tilde x, \tilde \eta) =\tilde x^T (B\overline H \tilde \eta+ B v_{2}) +
\tilde \eta^T  (\overline S \tilde \eta + v_{1})$.
Under the assumption of the skew-simmetry of $\overline{S}$, one obtains
$\dot{V}(\tilde x, \tilde \eta) =\tilde x^T B\overline H \tilde \eta+ \tilde x^T B v_{2} + \tilde \eta^T  v_{1}$.
Set
\be\label{v}
v_1 = -\overline H^T B^T\tilde x,\quad v_2 = -B^T \tilde x.
\ee
Observe that by the connectivity of the graph and the definition of $\tilde x$, $v_1 = -\overline H^T B^T x$ and  $v_2 = -B^T x$.
Then
$\dot{V}(\tilde x, \tilde \eta) =-|| B^T \tilde x||^2$.
Hence, $(\tilde x,\tilde \eta)$ is bounded. By La Salle's invariance principle and connectivity of the graph, the solutions to (\ref{cls.new.coord}) converge to the largest invariant set contained in $\{(\tilde x, \tilde \eta): B^T \tilde x=\mathbf{0}\}=\{(\tilde x, \tilde \eta): \tilde x\in {\cal R}(\mathbf{1}_n)\}$.\\
Observe that the system (\ref{cls.new.coord}) with the inputs $v$ as in (\ref{v}) becomes
\be\label{cls.new.coord.v}\ba{rcl}
\dot{\tilde x} &=&  - B B^T \tilde x +B\overline H \tilde \eta\\
\dot{\tilde \eta} &=& \overline S \tilde \eta-\overline H^T B^T\tilde x
\ea\ee
On this invariant set the system (\ref{cls.new.coord.v}) satisfies
\[
\ba{rcl}
\dot{\tilde x} &=&  B\overline H \tilde \eta\\
\dot{\tilde \eta} &=& \overline S \tilde \eta\\
\mathbf{0} &=& B^T \tilde x.
\ea
\]
Hence, $\tilde x=\mathbf{1}_n  \tilde x_\ast$. Replacing this expression in the equation for $\tilde x$ and pre-multiplying both sides by $\mathbf{1}_n^T$, one obtains $\dot {\tilde x}_\ast=0$, that is ${\tilde x}_\ast$ is a constant. Hence $\tilde x=x-x^w\to \mathbf{1}_n c$ for some constant $c$. Bearing in mind the expression of $x^w$ obtained in Lemma \ref{l1},  then one concludes that $x(t)\to \mathbf{1}_n (c' +\int_0^t \frac{\mathbf{1}_n^T Pw(s)}{n} ds)$ for some constant $c'$.
\end{proof}

\begin{remark}
In the case of balanced demand/supply, the state $x(t)$ converges to $\mathbf{1}_n c'$ for some constant $c'$.  Observe that $ \mathbf{1}^T \dot x =\mathbf{0}$, that is $\mathbf{1}^T x(t)= \mathbf{1}^T x(0)$. Hence, $\mathbf{1}^T x(0)=\lim_{t\to\infty} \mathbf{1}^T x(t)=nc'$ implies that $x(t)$ converges to $\mathbf{1}_n \frac{\mathbf{1}^T x(0)}{n}$. Hence under the effect of a  time-varying but balanced demand/supply all the components of the state $x(t)$ asymptotically converge to the average of the initial distribution of material at the nodes.
\end{remark}

Bearing in mind the block diagonal nature of the matrices $\overline S$, $\overline H$ and the definition $z=B^T x$, the dynamic feedback controller (\ref{ix})  can be decomposed  as the following set of dynamic feedback controllers at the  edges:
\be\label{dfc}
\ba{rcl}
\dot \eta_k &=& S^d \eta_k -H_{k}^T z_k\\
\lambda_{k} &=& H_k \eta_k-z_{k},\quad k=1,2,\ldots, n-1
\ea\ee
which only requires the knowledge of the difference between the quantities stored at the two nodes connected by the edge. As such the proposed controller (\ref{ix}) is fully distributed and solves the load balancing problem formulated in Section \ref{sec1}, with $\Phi_k =S^d, \Lambda_k=-H_k^T, \Psi_k= H_k, \Gamma_k=-1$. By Remark \ref{rem.H0}, for $k=n,n+1, \ldots, m$ for which $H_k=\mathbf{0}$ the edge controller becomes a static one, i.e.~$\lambda_{k} = -z_{k}$.


\setcounter{example}{0}

\begin{example}
(Cont'd) Assume that $d_1=\alpha+\beta\sin (\omega t+\varphi)$, with $\alpha>\beta>0$ and $d_2 =\alpha$. The supply is a periodic fluctuation around a constant value while the demand is a constant. Then the matrices $S^d$ and $\Gamma^d$ in (\ref{exos}) write as
\[
S^d=
\left(\ba{ccc}
0 & 0 & 0\\
0 & 0 & \omega\\
0 & - \omega & 0\\
\ea\right),
\quad \Gamma^d =
\left(\ba{ccc}
1 & 1 & 0\\
1 & 0 & 0\\
\ea\right). 
\] 
Let $\lambda^w_3=0$. Then, for $k=1,2$, the matrices $H_k$ which allow to reproduce $\lambda^w_k$ are
\[
H_1=\left(
\ba{ccc} \frac{2}{3} & \frac{1}{3} & 0\ea
\right),\quad 
H_2=\left(
\ba{ccc} \frac{1}{3} & -\frac{1}{3} & 0\ea
\right)
\] 
Then the controllers at the edges $1$ and $2$ are given by (\ref{dfc}) with $S^d$ and $H_k$ as above and 
\[
z_1=-x_1+x_2, \quad z_2= -x_2+x_3.
\] 
The controller at edge $3$ is the static control law $\lambda_3=-z_3=-(x_1-x_3)$. 
\end{example}

\begin{remark} {\bf (Passivity-based reinterpretation)}
The proof of Theorem \ref{th1} can be  reinterpreted as follows. In view of Lemma \ref{l1}, the system
\[\ba{rcl}
\dot {\tilde x} &=& B \tilde \lambda\\
z &=& B^T \tilde x
\ea\]
is the incremental model associated with system (\ref{flow.net}). Similarly, by Lemma \ref{l3}, system
\[\ba{rcl}
\dot {\tilde \eta} &=& \overline S \tilde \eta +\overline H^T \tilde v\\
\tilde u &=& \overline H \tilde \eta,
\ea\]
where $\tilde u = u-u^w$ and $u^w:=\overline H \eta^w$,
is the incremental model associated with the internal model
\[\ba{rcl}
\dot {\eta} &=& \overline S \eta +\overline H^T  v\\
u &=& \overline H \eta
\ea\]
The systems are passive with respect to the storage functions $V_1(\tilde x)=\frac{1}{2} \tilde x^T \tilde x$ and $V_2(\tilde \eta)= \frac{1}{2} \tilde \eta^T \tilde \eta$ provided that $S^d$ is skew symmetric. The negative feedback interconnection of the two systems, namely
\[\ba{rcl}
\tilde \lambda &=& \lambda_{ext}-\tilde u\\
\tilde v &=& u_{ext}+z,
\ea\]
is passive as well from the input $(\lambda_{ext}, u_{ext})$ to the output $(z, \tilde u)$. The output feedback
\[
\left(\ba{c}
\lambda_{ext}\\ u_{ext}
\ea\right)=
-
\left(\ba{cc}
K & 0 \\ 0 & 0
\ea\right)
\left(\ba{c}
z\\ \tilde u
\ea\right)
\]
gives asymptotic convergence of the closed-loop system to the largest invariant set where $z=0$.
\end{remark}

We discuss briefly the difficulties related to the presence of possible state and input constraints.\\
{\it State constraints.} Consider a variation of the model (\ref{flow.net}) in which the positivity constraint on the amount of material stored at the nodes is enforced. The model becomes
\[
\dot x = (B\lambda +Pw)_x^+
\]
where $(b_i\lambda +p_i w)_{x_i}^+$ is the $i$th component of the vector $(B\lambda +Pw)_x^+$ and
\[
(\zeta_i)_{x_i}^+ =
\left\{\ba{ll}
\zeta_i &\hspace{-2mm} \textrm{if ($x_i >0$) or ($i=0$ and $\zeta_i\ge 0$)}\\
0 & \hspace{-2mm} \textrm{if ($x_i=0$ {and} $\zeta_i< 0$)}
\ea\right.
\]
We consider the special case of balanced demand and supply, i.e.~$\mathbf{1}_n^T P w=0$. As a consequence, $Pw=-B\lambda^w$ and
\[
\dot x = \dot {\tilde x} = (B\tilde \lambda)_x^+.
\]
The function $V_1(\tilde x)=\frac{1}{2} \tilde x^T \tilde x$, with $\tilde x=  x-\mathbf{1} x_\ast^w$ and $x_\ast^w>0$, satisfies
\[
\dot{V_1}(\tilde x)=\tilde x^T (B\tilde \lambda)_x^+.
\]
Observe that $\tilde x^T (B\tilde \lambda)_x^+=\sum_{i=1}^n \tilde x_i (b_i\tilde \lambda)_{x_i}^+= \tilde x^T (B\tilde \lambda)$.
This shows that the system
\[\ba{rcl}
\dot {\tilde x} &=& (B\tilde \lambda)_x^+\\
z &=& B^T \tilde x
\ea\]
is passive and the arguments of the previous remark can be used. A formal analysis requires to take into account the discontinuity of the system. This is not pursued here for lack of space.

{\it Edge capacity constraints.} Constraints on the capacity of the edges can be modeled via a saturation function replacing $\lambda$ in (\ref{flow.net}) with  $\sat (\lambda)$. Here, $\sat (\lambda)=(\sat (\lambda_1)\ldots \sat(\lambda_m))^T$
and $\sat(\lambda_k)=\min\{|\lambda_k|, c\}\sign(\lambda_k)$. Following Lemma \ref{l1}, let $x^w, \lambda^w$ be such that
\[
\dot x^w = B\sat(\lambda^w) + Pw
\] 
and $M$  such that $\sat(\lambda^w) =Mw$. For the problem to be feasible restrict the set of initial conditions $w_0$ of the exosystem $\dot w=S^d w$ in such a way that $||Mw(t)||_\infty < c$ for all $t\ge 0$  (\cite{RDS:AI:TAC01}). Consider the  incremental model
\[\ba{rcl}
\dot{\tilde x} &=&  B[\sat(\lambda)-\sat(\lambda^w)]\\
&=&  B\sat(\lambda)-BMw
\ea
\]
To tackle the problem, we assume the 
scenario in which at each edge a dynamic observer provides $\hat w$ that converges to $w$ asymptotically (or at each edge $k$ there exists an estimator which generates a local estimate $\hat w_k$ of $w$). Consider then 
the control input
\[\ba{rcl}
\lambda &=&  -\mu(B^T x)+M\hat w\\
&=&  -\mu(B^T \tilde x)+M\hat w,
\ea\]
where $\mu:\R^m\to \R^m$ is a  map such that each component is a monotonically increasing function which is zero at the origin.
The incremental model writes as
\[\ba{l}
\dot{\tilde x} = B\sat( -\mu(B^T \tilde x)+M\hat w)-BMw\\
= B\sat( -\mu(B^T \tilde x)+Mw+M(\hat w-w))-BMw.
\ea\]
The right-hand side is bounded and the solutions exists for all $t\ge 0$. Suppose that each component of $\mu$ is a function whose range is within $[-\frac{c}{4}, \frac{c}{4}]$. Then after a finite time, $\sat( -\mu(B^T \tilde x)+Mw+M(\hat w-w))= -\mu(B^T \tilde x)+Mw+M(\hat w-w)$ and the incremental model evolves as 
\[\ba{rcl}
\dot{\tilde x} 
&=& -B \mu(B^T \tilde x)+BM(\hat w-w). 
\ea\]
Consider the projected state $y=Q \tilde x$, where $Q$ is an $(n-1)\times n$ matrix such that $Q\mathbf{1}_n=\mathbf{0}$, $QQ^T=I_{n-1}$ and $Q^T Q= I_n-\mathbf{1}_n\mathbf{1}_n^T/n$. It yields
\[\ba{rcl}
\dot{y} 
&=& -QB \mu(B^TQ^T y)+QBM(\hat w-w). 
\ea\]
The unforced system has a globally asymptotically stable equilibrium\footnote{Take $V(y)=\frac{y^T y}{2}$; then $\dot V(y)\le 0$  and $\dot V(y)=0$ is identically zero if and only if $B^TQ^T y=0$. This implies that $y=0$. In fact if this were not true, that is $B^TQ^T y=0$ and $y\ne 0$, then $ Q B B^TQ^T y=0$ as well and this would contradict that $y\ne 0$ since $ Q B B^TQ^T$ is a non singular matrix.}; moreover the forcing term is decaying to zero.  Since the response of the system is bounded then the state $y$ converges to the origin which implies that $\tilde x$ converges to ${\cal R}(\mathbf{1}_n)$. Then one can proceed as in the last part of the proof of Theorem \ref{th1}. 
The proposed solution relies on the existence of distributed estimators for $w$, whose actual design is left as a topic for future research.


\section{Flow networks with nonlinear dynamics at the nodes}\label{sec3}

In the  previous section, the dynamics describing the evolution of the storage variable at each node was given by
\be\label{node.dynamics.linear}
\dot x_i = b_i \lambda + p_i w,\quad i=1,2,\ldots, n
\ee
where $b_i$ and $p_i$ are the $i$th row of the incidence matrix $B$ and $P$ respectively. Consider now a different case of a flow network in which the way material accumulates at the node is described by a non-trivial dynamics, namely
\be\label{node.dynamics.nonlinear}
\dot x_i = f_i(x_i)+b_i \lambda + p_i w,\quad i=1,2,\ldots, n
\ee
with vector of measurements $y_i\in \R^m$ given by
\[
y_i = b_i^T x_i.
\]
The nonlinear system (\ref{node.dynamics.nonlinear}) allows us to put the results of the paper in a broader context and compare them with those in \cite{arcak.tac07}, \cite{MB:DZ:FA:CDC11} (see the end of the section). 
Observe that for $k=1,2,\ldots, m$,
$y_{ik}$ is either $x_i$, $-x_i$ or $0$. 
The sum of the outputs $y_i$ over all the nodes returns the vector of relative measurements $z$,
\[
z=B^T x= \dst\sum_{i=1}^n y_i.
\]
Each system
\be\label{node.dynamics}
\ba{rcl}
\dot x_i &=& f_i(x_i)+b_i \lambda + p_i w\\
y_i &=& b_i^T x_i,\quad i=1,2,\ldots, n
\ea
\ee
is assumed to be incrementally passive.
\begin{assumption}\label{a.inc.passivity}
There exists a regular\footnote{See \cite{PM:SCL08} for a definition.}  storage function $V_i:\R\times \R\times \R_+ \to \R_+$ such that
\[\ba{l}
\dst\frac{\partial V_i}{\partial t}+\dst\frac{\partial V_i}{\partial x_i} (f_i(x_i)+b_i \lambda + p_i w) +\\[2mm]
\dst\frac{\partial V_i}{\partial x'_i} (f_i(x'_i)+b_i \lambda' + p_i
w)
\le (y_i-y_i')^T (\lambda-\lambda').
\ea
\]
\end{assumption}


\begin{remark} {\bf (A class of incrementally passive systems)}
Consider the linear dynamics at the node
(\ref{node.dynamics.linear}) and the function $V_i=\frac{1}{2}(x_i-x_i')^2$. Then the right-hand side of the inequality above  becomes
\[\ba{l}
(x_i-x_i')(b_i \lambda + p_i w) -(x_i-x_i')(b_i \lambda' + p_i w)\\
= (x_i-x_i')b_i (\lambda-\lambda')\\
= (b_i^T(x_i-x_i'))^T(\lambda-\lambda')\\
= (y_i-y_i')^T(\lambda-\lambda')
\ea\]
which satisfies the dissipation inequality in Assumption \ref{a.inc.passivity}.  \\
Suppose that the dynamics $f_i$ are equal to $\nabla F_i$, with $F_i$ a twice continuously differentiable and concave function. Then  the static nonlinearity $-f_i(x_i)$ is incrementally passive, that is
\[
(x_i-x_i')(f_i(x_i)-f_i(x_i')) \le 0.
\]
As a matter of fact $f_i(x_i)-f_i(x_i')=\nabla F_i(x_i)-\nabla F_i(x_i')=\nabla^2 F_i(\xi_i)(x_i-x_i')$ for some $\xi_i$ lying in the segment connecting $x_i, x_i'$. By concavity, $\nabla^2 F_i(\xi_i)\le 0$ and therefore $(x_i-x_i')(f_i(x_i)-f_i(x_i'))\le 0$.
Hence any system (\ref{node.dynamics})  with $f_i(x_i)=\nabla F_i(x_i)$ and $F_i$ defined as before satisfies Assumption \ref{a.inc.passivity}.
\end{remark}

Lemma \ref{l1} is replaced by the following:
\begin{lemma}\label{re2}
For each $i=1,2,\ldots, n$, for each $w$ solution to $\dot w = S^d w$, there exist a function $\lambda^w:\R_+\to \R^m$ and continuously differentiable bounded functions $x^w_i:\R_+\to \R$ that satisfy
\be\label{re.eq.nonl}\ba{rcl}
\dot x_i^w &=& f_i(x_i^w)+b_i \lambda^w + p_i w,\quad i=1,2,\ldots, n\\
0 &=& \dst\sum_{i=1}^n b_i^T x^w_i
\ea\ee
only if there exists a solution $x^w_\ast:\R_+\to \R$ defined for all $t\ge 0$ to
\be\label{exws}
\dot{x}^w_\ast = \dst\frac{\mathbf{1}_n^T f(x^w_\ast)}{n}+\dst\frac{\mathbf{1}_n^T P w}{n},
\ee
where $f(x)=(f_1(x)\ldots f_n(x))^T$. 
If this is the case, then 
\[\ba{rcl}
x^w_i&=&x^w_\ast,\;i=1,2,\ldots, n,\\
\lambda^w&=&\left(\ba{c}
{\lambda^w_a}\\
{\lambda^w_b}
\ea\right)=
\left(\ba{c}
M_1 f(x^w_\ast)+M_2 w\\
\mathbf{0}\\
\ea\right)\ea\]
with $\lambda_a^w\in \R^{n-1},  \lambda_b^w\in \R^{m-n+1}$, and $M_1, M_2$ suitable matrices.
\end{lemma}

\begin{proof}
The second equality in (\ref{re.eq.nonl}) and Assumption \ref{connectivity} implies that $x^w_i=x^w_\ast$ for all $i$. Replacing the latter in the first equality implies that necessarily $x^w_\ast$ must be a solution of the inhomogeneous differential equation
\[
\dot{x}^w_\ast = \dst\frac{\mathbf{1}_n^T f(x^w_\ast)}{n}+\dst\frac{\mathbf{1}_n^T P w}{n}.
\]
Suppose that a solution $x^w_\ast$ exists for all $t$ and let $x_i^w=x^w_\ast$ for each $i$. Then the second equation in (\ref{re.eq.nonl}) is satisfied by the connectivity of the graph and the properties of the incidence matrix.  Since $x_i^w={x}^w_\ast$ for all $i$, it is seen that  the first equation in (\ref{re.eq.nonl}) is satisfied if and only if there exists $\lambda^w$ such that
\[
Y[f(x^w_\ast)-Pw]=B\lambda^w, \quad \textrm{with}\;Y=\dst\frac{\mathbf{1}_n\mathbf{1}_n^T}{n}-I_n.
\]
By connectivity of the graph, the rank of $B$ is $n-1$. If the graph has $n-1$ edges (i.e.~ it is a tree), then $B$ is full-column rank and, provided that      a solution $\lambda^w$ to the previous equation exists, it is given by $\lambda^w=(B^T B)^{-1}B^T Y(f(x^w_\ast)-Pw)$.  If the graph has more than $n-1$ edges, then  without loss of generality (up to a relabeling of the edges of the graph) we can partition $B$ as $(B_a\; B_b)^T$ with $B_a$ full-column rank. Then, provided that a solution  to the previous equation exists, it is given by  $\lambda^w_a=(B_a^T B_a)^{-1}B_a^T[Y(f(x^w_\ast)-Pw)-B_b\lambda_b^w]$. One particular solution is obtained for $\lambda_b^w=0$ and $\lambda^w_a=(B_a^T B_a)^{-1}B_a^T Y(f(x^w_\ast)-Pw)$.
\end{proof}

\begin{remark}
If the inflow and outflow are balanced, i.e.~$\mathbf{1}_n^T P w=0$, then the solution $x^w_\ast$ to  (\ref{exws}) exists for all $t$ and is bounded. In fact, consider the system
\[
\dot{y} = \dst\frac{\mathbf{1}_n^T f(y)}{n}
\]
and the radially unbounded function $V(y)=\frac{1}{2} y^2$. Then
\[
\dot{V}(y)= y\dst\frac{\mathbf{1}_n^T f(y)}{n}=\dst\sum_{i=1}^n \frac{y f_i(y)}{n}.
\]
By the incremental passivity property of $-f_i$, $y f_i(y)\le 0$ for all $i$ and this implies $\dot{V}(y)\le 0$.  Hence every solution to the system above is bounded and so is $x^w_\ast$.
\end{remark}

\begin{remark}
In the case the dynamics at the nodes are all the same, i.e.~$f_i=f_j$ for all $i, j$, then the expression  of $\lambda^w$ simplifies as
\[
\lambda^w=\left(\ba{c}
{\lambda^w_a}\\
{\lambda^w_b}
\ea\right)=
\left(\ba{c}
M_2 \\
\mathbf{0}\\
\ea\right) w.
\]
This descends from the proof, since by definition of the matrix $Y$,
$Y f(x^w_\ast)=\mathbf{0}$.
\end{remark}

In the remaining of the section we assume that a solution to (\ref{re.eq.nonl}) exists. \\
The parallel interconnection of the $n$ subsystems (\ref{node.dynamics}) with input $\lambda$ and output
$z=\sum_{i=1}^n y_i$ returns an incrementally passive systems. 
Formally

\begin{lemma}\label{l4}
The parallel interconnection
\[
\ba{rcl}
\dot x_1 &=& f_1(x_1)+b_1 \lambda + p_1 w\\
&\ldots & \\
\dot x_n &=& f_n(x_n)+b_n \lambda + p_n w\\
z &=& \dst\sum_{i=1}^n b_i^T x_i,
\ea
\]
denoted as
\be\label{flow.net.big}
\ba{rcl}
\dot x &=& f(x)+ B \lambda + P w\\
z &=& B^T x
\ea
\ee
is such that the storage function $V(x,x')=\sum_{i=1}^n V_i(x_i,x'_i)$ satisfies
\[\ba{r}
\dst\frac{\partial V}{\partial x} (f(x)+ B \lambda + Pw) +
\dst\frac{\partial V}{\partial x'} (f(x')+B \lambda' + Pw)\\
\le (z-z')^T (\lambda-\lambda').
\ea\]
\end{lemma}
The proof is straightforward and is omitted. 
%
%
%
Consider now systems of the form
\be\label{nonl.im}
\ba{rcl}
\dot \eta_k &=& \phi_k(\eta_k, v_k)\\
u_k &=& \psi_k(\eta_k), \quad k=1,2,\ldots, n-1,
\ea
\ee
with the following two additional properties:
\begin{assumption}\label{a.ip.im}
For each $k=1,2,\ldots, n-1$, there exists  regular functions $W_k(\eta_k, \eta_k')$ such that
\[
 \dst\frac{\partial W_k}{\partial \eta_k} \phi(\eta_k, v_k) +
 \dst\frac{\partial W_k}{\partial \eta'_k} \phi(\eta'_k, v'_k) \le
 (u_k-u_k') (v_k-v_k').
\]
\end{assumption}

\begin{assumption}\label{imp}
For each $k=1,2,\ldots, n-1$, for each $w$ solution to $\dot w=S^d w$, there exists   a bounded solution $\eta^w_k$  to $\dot \eta_k =\phi_k(\eta_k, 0)$ such that $\lambda_k^w=\psi_k(\eta_k^w)$.
\end{assumption}

Assume that the system
\[\ba{rcl}
\dot{\eta}_{ka}^w &=& \dst\frac{\mathbf{1}_n^T f({\eta}_{ka}^w)}{n}+\dst\frac{\mathbf{1}_n^T P {\eta}_{kb}^w}{n}\\
\dot{\eta}_{kb}^w &=& S^d {\eta}_{kb}^w
\ea\]
is forward complete. Initialize the system as ${\eta}_{ka}^w(0)= x^w_\ast(0)$ and ${\eta}_{kb}^w(0)= w(0)$. Then ${\eta}_{ka}^w(t)=x^w_\ast(t)$ and ${\eta}_{kb}^w(t)= w(t)$ for all $t\ge 0$. Hence
$\lambda^w_k=M_{1k} f({\eta}_{ka})+M_{2k} {\eta}_{kb}$, $k=1,2,\ldots, n-1$, where $M_{1k}$ and $M_{2k}$ are the $k$th rows of $M_1$ and $M_2$ respectively. On the other hand, $\lambda^w_k=0$, $k=n, n+1, \ldots,
m$. 
An expression for $\phi_k, \psi_k$, $k=1,2,\ldots, n-1$ is
\[
\phi_k(\eta_k,0)= \left(\ba{c}
\dst\frac{\mathbf{1}_n^T f({\eta}_{ka})}{n}+\dst\frac{\mathbf{1}_n^T P {\eta}_{kb}}{n}\\
S^d{\eta}_{kb}
\ea\right),
\]
$\psi_k(\eta_k)=M_{1k} f({\eta}_{ka})+M_{2k} {\eta}_{kb}$.\\
In the special case of nodes with the same dynamics ($f_i=f_j=\bar f$ for all $i,j$) $\psi_k(\eta_k)$ simplifies as $M_{2k} {\eta}_{kb}$ and a system that satisfies Assumptions \ref{a.ip.im} and \ref{imp} is 
\[\ba{rcl}
\dot \eta_k &=& S^d {\eta}_{k} +M_{2k}^T v_{k}\\
u_k &=& M_{2k} {\eta}_{k},
\ea\]
with storage function $W_k(\eta_k)=\frac{1}{2} \eta_{k}^T\eta_{k}$. 
%
%
%
%
%
Collect the systems (\ref{nonl.im}) into a system with state variable $\eta=(\eta_1^T\ldots\eta_{n-1}^T)^T$, input $v=(v_1\ldots v_m)^T$ and
output $u=(u_1\ldots u_m)^T$, namely
\be\label{im.big}
\ba{rcl}
\dot \eta &=& \Phi(\eta,v)\\
u &=& \Psi(\eta)
\ea\ee
with $\Phi(\eta,v)=(\phi_1^T\ldots \phi_{n-1}^T)^T$, $\Psi(\eta)=(\psi_1\ldots \psi_{n-1}\; \mathbf{0}^T)^T$. The system is incrementally passive from $v$ to $u$ with storage function $W(\eta,\eta')=\sum_{k=1}^{n-1} W_k(\eta_k,\eta_k')$.
\\
The following holds:
\begin{theorem}
Let Assumptions \ref{connectivity}-\ref{imp} hold. Suppose that a solution to (\ref{re.eq.nonl}) exists and $x_\ast^w$ is bounded. 
Consider the systems (\ref{flow.net.big}), with input $\lambda$ and output $z$, and (\ref{im.big}), with input $v$ and output $u$, interconnected via the relations
$v=-z+v_{ext},\quad \lambda= u+\lambda_{ext}$.\\
The interconnected system is incrementally passive from the input $(\lambda_{ext}^T\; v_{ext}^T)^T$ to the output $(z^T\; u^T)^T$. Moreover, the feedback $(\lambda_{ext}^T\; v_{ext}^T)^T=(-K z^T\; \mathbf{0}^T)^T$, with $K$ a positive definite diagonal matrix,  guarantees $\lim_{t\to+\infty} z(t)=\mathbf{0}$.
\end{theorem}

\begin{proof}
The feedback interconnection of incrementally passive systems is incrementally passive (\cite{PM:SCL08}, Lemma 1). Hence
\[
\ba{rcl}
\dot x &=& f(x)+ B \lambda + P w\\
z &=& B^T x\\[2mm]
\dot \eta &=& \Phi(\eta,v)\\
u &=& \Psi(\eta)\\[2mm]
\lambda &=& u+\lambda_{ext}\\
v&=& -z+v_{ext}
\ea
\]
is incrementally passive from the input $(\lambda_{ext}^T\; v_{ext}^T)^T$ to the output $(z^T\; u^T)^T$. The storage function $U$ is given by the sum $V+W$ where $V, W$ are the functions defined above (in Lemma \ref{l4} and after (\ref{im.big}), respectively). \\
Let $\lambda_{ext}=-K z$, $v_{ext}=\mathbf{0}$. The system becomes
\[
\ba{rcl}
\dot x &=& f(x)+ B (\Psi(\eta)-K B^T x) + P w\\
\dot \eta &=& \Phi(\eta,-z)\\
z &=& B^T x
\ea
\]
For a given solution $w$ to $\dot w = S^d w$, let $x^w$, $\lambda^w$ be as in  Lemma \ref{re2} and $\eta^w$ as in Assumption \ref{imp}. The functions $x^w$ and $\eta^w$ are a solution to the equations above with input $(\lambda_{ext}^T\; v_{ext}^T)^T=\mathbf{0}$ and output $(z^T\; u^T)^T=(\mathbf{0}^T {\lambda^w}^T)^T$.
 In fact
\[
\ba{rcl}
\dot x^w
&=& f(x^w)+ B \Psi(\eta^w) + P w\\
&=& f(x^w)+ B \lambda^w + P w\\
\dot \eta^w &=& \Phi(\eta^w,\mathbf{0} )\\
\mathbf{0} &=& B^T x_w.
\ea
\]
As in \cite{PM:SCL08}, by the incremental passivity of the feedback system and the existence of a solution $(x^w, \eta^w)$ of the feedback system such that $z(t)=\mathbf{0}$, any other solution $(x,\eta)$ with input $(\lambda_{ext}^T\; v_{ext}^T)^T=(-Kz^T\; \mathbf{0}^T)^T$
satisfies
\[\ba{l}
\dot V((x,\eta), (x^w,\eta^w))\le \\
((z^T\; u^T)-(\mathbf{0}^T\;{\lambda^w}^T))
\left(\ba{c}-K z(t)\\ \mathbf{0}\ea\right)=
-  z^T K z.\ea
\]
Bearing in mind the regularity of $U$ and boundedness of $x^w$, this yields boundedness of $x$. In view of the time-varying nature of the system, to infer convergence of $z$ to zero, one can resort to Barbalat's  lemma. This guarantees convergence under the assumption that $\dot z$ is bounded. This in turn requires  $\dot w$ bounded, which is the case here since $S$ is  skew symmetric.
\end{proof}


\begin{corollary}\label{c1}
If  (i) $f_i=\bar f$ for all $i=1,2,\ldots, n$, (ii) there exists a twice continuously differentiable convex function $F(x)$ such that  $\nabla F(x)=\bar f(x)$ and (iii) $\mathbf{1}_n^T P w=0$ for all $t\ge 0$, then the controllers
\[\ba{rcl}
\dot{\eta}_{k} &=&
S^d{\eta}_{k}- M_{2k}^T z_k
\\
\lambda_k &=& M_{2k} {\eta}_{k}-z_k, ,\quad k=1,2,\ldots, n-1,
\ea\]
and  $\lambda_k = -z_k$, $k=n,n+1,\ldots, m$, guarantee $\lim_{t\to+\infty} z(t)=\mathbf{0}$.
\end{corollary}

The closed-loop system given  in the corollary above takes the form 
\[\ba{llllll}
\dot x &=& \nabla F(x)+B\lambda +Pw, &
z &=& B^T x\\
\dot \eta &=& \overline S \eta - M_2^T z,
&
\lambda &=& M_2\eta - z\;,
\ea\]
where we are assuming that $m=n-1$ for the sake of simplicity.  
This system can be compared with similar ones appeared in the recent literature 
(\cite{arcak.tac07}, \cite{MB:DZ:FA:CDC11}), where models of the form
\[\ba{llllll}
\dot x &=& \nabla F(x)+B\lambda,&
z &=& B^T x\\
\dot \eta &=& z,&
\lambda &=& -\psi(\eta)\\
\ea\]
with $\psi$ a non-decreasing monotonic non-linearity (such as a saturation function), were studied.  The presence of the non-trivial dynamics  $\overline S$ in our controller is due to the time varying-nature of the external input. 
In \cite{arcak.tac07}, $\nabla F(x)$ has a unique equilibrium at the origin and the system $\dot x = \nabla F(x)+B\lambda$ is strictly passive. In 
\cite{MB:DZ:FA:CDC11} it is shown that if the components of the vector field $\nabla F(x)$ have different equilibria, $\nabla F(x)$ is strongly concave  and $\psi$ introduces saturation constraints, then the system's response  exhibits state clustering. 

\section{Conclusions}
We have presented an internal model approach to the problem of balancing demand and supply in a class of distribution networks. Extensions to nonlinear systems have also been discussed. Further research will focus on a detailed investigation of state and input constraints and more complex models of demand and supply. The fulfillment of the internal model principle has to be understood for more general classes of nonlinear systems than those in Corollary \ref{c1}. This will shed light on the relation between the results in this paper and the saddle-point perspective of \cite{MB:DZ:FA:CDC11}. Compared with other papers  where the robustness to time-varying inputs is studied  using a frequency domain approach (\cite{bai.et.al.cdc10}), our state space approach allows us to consider more general classes of cooperative control systems.

\bibliographystyle{plain}
\bibliography{QUICK-BIBLIO_16Aug11}

\end{document}